\documentclass{amsart}
\usepackage{amssymb}
\usepackage{fullpage}
\usepackage{amsrefs}
\usepackage[utf8]{inputenc}
\usepackage{xy}
\xyoption{all}
\usepackage{mathtools}
\usepackage{graphicx}

\newcommand{\C}{{\mathbb C}}
\newcommand{\F}{{\mathbb F}}

\newcommand{\R}{{\mathbb R}}

\newcommand{\tr}{{\rm t}}
\newcommand{\trace}{\operatorname{trace}}

\theoremstyle{plain}
\newtheorem{theorem}{Theorem}
\newtheorem{proposition}[theorem]{Proposition}
\newtheorem{lemma}[theorem]{Lemma}
\newtheorem{corollary}[theorem]{Corollary}

\theoremstyle{definition}
\newtheorem{definition}[theorem]{Definition}

\newcommand{\OO}{\operatorname{O}}
\newcommand{\U}{\operatorname{U}}
\newcommand{\E}{\operatorname{E}}
\newcommand{\FF}{\operatorname{F}}

\newcommand{\real}{\operatorname{Re}}

\title{Bi-Lipschitz Quotient embedding for\\Euclidean Group actions on Data}
\author{Harm Derksen}
\address{Department of Mathematics\\Northeastern University\\360 Huntington Ave\\Boston, MA 02115}
\email{ha.derksen@northeastern.edu}
\thanks{The author was partially supported by NSF grant DMS 2147769.}

\begin{document}

\begin{abstract}
    If $G$ is a group acting on $\R^n$ by euclidean transformations, then (under mild assumptions) we 
    can define a metric on the orbit space $\R^n/G$. Recent papers have studied the question when there exists a bi-Lipschitz embedding $\R^n/G\hookrightarrow \R^m$ for some $m$. We will construct such an embedding with distortion $\sqrt{2}$ when $G$ is the orthogonal group or euclidean group 
    acting on $\ell$-tuples of vectors in $\R^n$. We also obtain similar results for the unitary group and complex euclidean group acting on $\ell$-tuples of vectors in $\C^n$. We work out the example of the euclidean group acting on triangles in the plane in more detail. The map that sends a triangle to the triple of sidelengths is not bi-Lipschitz. Experiments show that  our bi-Lipschitz construction is less sensitive to noise than the sidelengths map for the problem of matching a noisy sample triangle to a triangle in a database.

\end{abstract}
\maketitle

\section{Introduction}
\subsection{The orbit problem}
Suppose that $G$ is a group acting on $\R^n$. The orbit of a vector $a\in \R^n$ is $G\cdot a=\{g\cdot a\mid g\in G\}$.
The orbit space $\R^n/G$ is the set of all $G$-orbits in $\R^n$
and we define the quotient map $\pi:\R^n\to \R^n/G$ by
$\pi(a)=G\cdot a\in \R^n/G$. 
Given two vectors $a,b\in \R^n$ one can ask whether they lie in the  same $G$-orbit. We call this the {\em orbit problem}.

Many problems can be formulated as an orbit problem. Take for example the graph isomorphism problem. Two graphs ${\mathcal A}$ and ${\mathcal B}$
on $n$ vertices correspond to adjacency matrices $A$ and $B$ in $\R^{n\times n}\cong \R^{n^2}$ respectively.
The graphs ${\mathcal A}$ and ${\mathcal B}$ are isomorphic
if and only if $A$ and $B$ lie in the same $S_n$
orbit, where the symmetric group $S_n$ acts
 on $\R^{n\times n}$ by simultaneously permuting rows and columns.

 % Another example that we will discuss in more detail is the action of the euclidean group $\OO(n)$ on $\ell$-tuples of vectors. We can think of an $\ell$-tuple of vectors
 % as an $n\times \ell$ matrix $A=[a_1\ a_2\ \cdots\ a_\ell]\in \R^{n\times \ell}$
 % and $\OO(n)$ acts on $\R^{n\times \ell}$ by left multiplication.

\subsection{An approach using invariant theory}
One approach to the orbit problem  uses invariants.
\begin{definition}
A feature map $\phi:\R^n\to \R^m$ is called {\em invariant}
if $G\cdot a=G\cdot b$ implies $\phi(a)=\phi(b)$. It is called {\em orbit separating} if $\phi(a)=\phi(b)$ implies $G\cdot a=G\cdot b$.
\end{definition}
If $\phi:\R^n\to\R^m$ is an orbit separating invariant feature map,
then $G\cdot a=G\cdot b$ $\Leftrightarrow$ $\phi(a)=\phi(b)$. So we can easily test whether two vectors lie in the same orbit. If a feature map $\phi:\R^n\to \R^m$ is $G$-invariant,
then this map factors $\phi=\overline{\phi}\circ \pi$
where $\overline{\phi}(G\cdot a)=\phi(a)$ for $a\in \R^n$.
Now $\phi$ is orbit separating if and only of $\overline{\phi}$ is injective.

We can construct an orbit separating invariant feature map using invariant theory. In invariant theory one studies polynomial functions $\R^n\to \R$ that are invariant. Let $\R[x_1,x_2,\dots,x_n]$ be the
polynomial ring, where $x_1,x_2,\dots,x_n$ are the coordinate functions on $\R^n$. Any element $f(x)=f(x_1,x_2,\dots,x_n)\in\R[x_1,x_2,\dots,x_n]$
is a polynomial function $\R^n\to \R$.
 The invariant ring $\R[x_1,x_2,\dots,x_n]^G$
consists of all {\em invariant} polynomials in $\R[x_1,x_2,\dots,x_n]$. 
If $G$ is a closed subgroup of the orthogonal group $\OO(n)$, then $G$ is compact. By Hilbert's Finiteness Theorem (see~\cites{Hilbert1890,Hilbert1893}) there exist finitely many polynomials $f_1(x),f_2(x),\dots,f_m(x)$ that generate
the invariant ring $\R[x_1,x_2,\dots,x_n]^G$. This means
that every $G$-invariant is of the form $p(f_1(x),f_2(x),\dots,f_m(x))$ where $p(y_1,y_2,\dots,y_m)\in \R[y_1,y_2,\dots,y_m]$ is some polynomial in $m$ variables. We now can define a feature map $\phi:\R^n\to \R^m$ by $\phi(x)=(f_1(x),f_2(x),\dots,f_m(x))$. 
The map $\phi$ is invariant,  because $G$ is invariant.
The map $\phi$ is also orbit separating (see~\cite[Theorem 3.12]{Schwarz2004}).

% If we go back to the example of $\OO(n)$ acting on $\R^{n\times \ell}$. The first fundamental theorem of invariant theory for $\OO(n)$ states that the invariant
% ring is generated by the polynomials $f_{i,j}:[a_1\ a_2\ \cdots\ a_\ell]\mapsto \langle a_i,a_j\rangle$,
% where $\langle\cdot,\cdot\rangle$ is the inner product on $\R^n$. If we set $\phi=(f_{i,j})_{i,j=1}^{\ell}$
% then $\phi:\R^n\to \R^{\ell\times \ell}\cong \R^{\ell^2}$
% is given by $\phi(A)=A^\tr A$.
\subsection{Bi-Lipschitz feature maps}
For applications, rather than testing whether two vectors $a,b\in \R^n$ we may want to see how close the orbits $G\cdot a$ and $G\cdot b$ are to each other. On $\R^n$ we have a euclidean norm given by  $\|a\|=\sqrt{\langle a,a\rangle}$, where $\langle a,b\rangle$ is the inner product between vectors $a,b\in \R^n$. The euclidean distance  is
given by $d(a,b)=\|a-b\|$.
Let $\OO(n)$ be the orthogonal group acting on $\R^n$. The euclidean group $\E(n)$ consists of all maps $\R^n\to \R^n$ that preserve distance. More concretely,  $\E(n)$ consists of all pairs
 $g=(P,q)$ where $P\in \OO(n)$ and $q\in \R^n$.
An element $g=(P,q)\in \R^n$ acts on $\R^n$ by $g\cdot a=Pa+q$,
and we have $d(g\cdot a,g\cdot b)=d(a,b)$ for all $g\in \E(n)$ and $a,b\in \R^n$. Suppose that $G$ is a closed subgroup  $\E(n)$.
The distance between the orbits $G\cdot a$ and $G\cdot b$ is
\begin{equation}
d(G\cdot a,G\cdot b)=\min_{g,h\in G}d(g\cdot a,h\cdot b)=
\min_{g\in G}(g\cdot a,b).
\end{equation}
This distance is a well defined metric (see Proposition~\ref{prop:ismetric}) on $\R^n/G$.
We will also write $d_G(a,b)$ for $d(G\cdot a,G\cdot b)$.
We can think of $d_G$ is a pseudo-metric on $\R^n$
with the property that $d_G(a,b)=0$ $\Leftrightarrow$ $G\cdot a=G\cdot b$.
We would like $d(G\cdot a,G\cdot b)$ to be close
to the distance $d(\phi(a),\phi(b))$. So we make the following definition:
\begin{definition}
   A function $\phi:\R^n\to\R^m$ is called $G$-bi-Lipschitz if there exist constants $C_1,C_2>0$ such that for all $v,w\in \R^n$ we have
    \begin{equation}\label{eq:biLipschitz}
        C_1 d_G(a,b)\leq d(\phi(a),\phi(b))\leq C_2 d_G(a,b).
    \end{equation}
    If $C_1$ is the largest constant and $C_2$ is the smallest constant for which (\ref{eq:biLipschitz}) holds, then the ratio $C_2/C_1$ is called the {\em distortion}.
\end{definition}
Recall that a $G$-invariant map $\phi:\R^n\to \R^m$ factors
through the quotient $\pi:\R^n\to\R^n/G$ as  $\phi=\overline{\phi}\circ \pi$ with $\overline{\phi}:\R^n/G\to \R^m$. Now $\phi$ is $G$-bi-Lipschitz if and only if $\overline{\phi}$ is bi-Lipschitz.

The feature maps constructed from invariant theory typically do not have the $G$-bi-Lipschitz property. If the invariant ring $\R[x_1,x_2,\dots,x_n]^G$ is generated by 
invariants $f_1(x),f_2(x),\dots,f_m(x)$,
and any of these generators has degree $>1$, then the map $\phi:\R^n\to\R^m$ is not Lipschitz and therefore not $G$-bi-Lipschitz. A $G$-bi-Lipschitz map 
$\phi:\R^n\to \R^m$ has to be Lipschitz and therefore continuous, but often it cannot be differentiable (see~\cite[Section 5]{CIM2023}).
Bi-Lipschitz feature maps were studied in~\cites{CIM2023,CIMP2022,CCC2020,CCC2024,BT2023a,BT2023b,MQ2024}. 
It was shown in~\cite{CIMP2022} that a $G$-bi-Lipschitz feature map exists when $G$ is a finite subgroup of $\OO(n)$. We conjecture that a bi-Lipschitz feature map exists for any closed subgroup of the euclidean group $\E(n)$.
Some partial results were found in~\cite{MQ2024}
for infinite closed subgroups of $\OO(n)$.

\subsection{Orthogonal and euclidean group actions on tuples of vectors}
We consider the action of the euclidean group $\E(n)$ on $\ell$-tuples of vectors in $\R^n$.
We can view an $\ell$-tuple $(a_1,a_2,\dots,a_\ell)$ of vectors in $\R^n$
as an $n\times \ell$ matrix $A=[a_1\ a_2\ \cdots\ a_\ell]\in \R^{n\times \ell}$. The bilinear form on $\R^{n\times \ell}$
is given by $\langle A,B\rangle=\trace(AB^\tr)$ and the norm is given by $\|A\|=\sqrt{\langle A,A\rangle}=\sqrt{\trace(AA^\tr)}$.
An element $g\in \E(n)$
acts simultaneously on the vectors by left multiplication:
$$
g\cdot V=
g\cdot \begin{bmatrix}
    a_1 & a_2 & \cdots & a_\ell
\end{bmatrix}=\begin{bmatrix}
    g\cdot a_1 & g\cdot a_2 & \cdots & g\cdot a_\ell
\end{bmatrix}
$$
for any $g\in \E(n)$.

Let us consider the action of the orthogonal group $\OO(n)$ on $\R^{n\times \ell}$. If we follow the Invariant Theory approach we first find generating invariants. The first fundamental theorem of Invariant Theory for $\OO(n)$ states that the invariant ring is generated by all $f_{i,j}$, $1\leq i,j\leq \ell$,  where $f_{i,j}(a_1,a_2,\dots,a_\ell)=\langle a_i,a_j\rangle$ (see~\cite{GW}[4.4.2]).
We define $\phi:\R^{n\times \ell}\to \R^{\ell\times \ell}\cong \R^{\ell^2}$ by
$$
\phi(A)=A^\tr A=\begin{bmatrix}
    f_{1,1}(A) & f_{1,2}(A) & \cdots & f_{1,\ell}(A)\\
    f_{2,1}(A) & f_{2,2}(A) & \cdots & f_{2,\ell}(A)\\
    \vdots & \vdots & & \vdots\\
    f_{\ell,1}(A) & f_{\ell,2}(A) & \cdots & f_{\ell,\ell}(A)
\end{bmatrix}.
$$
Because the invariants $f_{i,j}$ have degree $2$, the map $\phi$ will not be Lipschitz. However, we can modify it so that it will become $\OO(n)$-bi-Lipschitz. The matrix $A^\tr A$ is positive semi-definite
and there is a unique positive semi-definite matrix $\sqrt{A^\tr A}$ whose square is $A^\tr A$. The following theorem gives an $\OO(n)$-bi-Lipschitz feature map and is the main result of this paper:
\begin{theorem}\label{theo:3}
Define $\phi:\R^{n\times \ell}\to \R^{\ell\times \ell}$ by
$\phi(A)=\sqrt{A^\tr A}$. Then we have
$$
d_{\OO(n)}(A,B)\leq d(\phi(A),\phi(B)) \leq \sqrt{2}d_{\OO(n)}(A,B)
$$
for all $A,B\in \R^{n\times \ell}$.
The image of $\phi$ is contained in the subspace $S^2(\R^\ell)\subseteq \R^{\ell\times \ell}$ of 
dimension ${\ell+1\choose 2}$ consisting of all
symmetric matrices. 
\end{theorem}
Using this theorem we can also give a bi-Lipschitz feature map for the action of the euclidean group:
\begin{theorem}\label{theo:4}
Define $\psi:\R^{n\times \ell}\to \R^{\ell\times \ell}$ by
$$
\psi\left(\begin{bmatrix}
    a_1 & a_2 & \cdots & a_\ell\end{bmatrix}\right)=\phi\left(\begin{bmatrix}
        a_1-\overline{a},a_2-\overline{a},\cdots,a_{\ell}-\overline{a}
        \end{bmatrix}\right),
$$
where $\overline{a}=\frac{1}{\ell}(a_1+a_2+\cdots+a_{\ell})$ is the average and $\phi$ is as in Theorem~\ref{theo:3}. Then we have
$$
d_{\E(n)}(A,B)\leq d(\psi(A),\psi(B)) \leq \sqrt{2}d_{\E(n)}(A,B)
$$
for all $A,B\in \R^{n\times \ell}$.
The image of $\psi$ lies in a subspace of $\R^{\ell\times \ell}$ of dimension ${\ell\choose 2}$.

\end{theorem}
We also have complex versions of Theorems~\ref{theo:3} and~\ref{theo:4}. For complex vectors $a,b\in \C^n$ we have
an hermitian inner product $\langle a,b\rangle\in \C$. 
Let $\U(n)$ be the unitary group of all linear maps preserving the hermitian inner product. Let $\FF(n)$ be the complex euclidean group consisting of all pairs $g=(P,q)$ where $P\in \U(n)$ and $q\in \C^n$.
An element $g=(P,q)$ acts on $\C^n$ by $g\cdot a=Pa+q$.
On the space $\C^{n\times \ell}$ the hermitian inner product is given by $\langle A,B\rangle=\trace(AB^\star)$ where $B^\star$ is the complex conjugate transpose matrix of $B$. We have $\|A\|=\sqrt{\langle A,A\rangle}=\sqrt{\trace(AA^\star)}$.

\begin{theorem}\label{theo:5}
    Define $\phi_{\C}:\C^{n\times \ell}\to \C^{\ell\times \ell}$ by $\phi(A)=\sqrt{A^\star A}$. Then we have
    $$
d_{\U(n)}(A,B)\leq d(\phi_\C(A),\phi_\C(B)) \leq \sqrt{2}d_{\U(n)}(A,B)
$$
for all $A,B\in\C^{n\times \ell}$.
The image of $\phi_\C$ is contained in the subspace ${\mathcal H}(\C^\ell)\subseteq \C^{\ell\times \ell}$  of real dimension $\ell^2$ consisting 
of all hermitian matrices.
\end{theorem}
\begin{theorem}\label{theo:6}
Define $\psi_\C:\C^{n\times \ell}\to \C^{\ell\times \ell}$ by
$$
\psi_\C\left(\begin{bmatrix}
    a_1 & a_2 & \cdots & a_\ell\end{bmatrix}\right)=\phi_\C\left(\begin{bmatrix}
        a_1-\overline{a},a_2-\overline{a},\cdots,a_{\ell}-\overline{a}
        \end{bmatrix}\right),
$$
where $\overline{a}=\frac{1}{\ell}(a_1+a_2+\cdots+a_{\ell})$ is the average and $\phi_\C$ is as in Theorem~\ref{theo:5}. Then we have
$$
d_{\FF(n)}(A,B)\leq d(\psi_\C(A),\psi_\C(B)) \leq \sqrt{2}d_{\FF(n)}(A,B)
$$
for all $A,B\in \R^{n\times \ell}$. The image of $\psi_\C$ is contained in a subspace of $\C^{\ell\times \ell}$ of real dimension $(\ell-1)^2$.
\end{theorem}

The $G$-bi-Lipschitz feature maps that we construct
for $G=\OO(n),\E(n),\U(n),\FF(n)$ map from a space of dimension $O(n\ell)$ to a space of dimension $O(\ell^2)$
(here $O(\cdot)$ is big-O notation, not the orthogonal group).
The following theorem shows us that it is possible
to construct a $G$-bi-Lipschitz map to a space of dimension $O(n\ell)$. The construction in the theorem is explicit,
even though we do not have any explicit upper bounds for the distortion of these $G$-bi-Lipschitz maps.

\begin{theorem}\label{theo:bi-Lipschitz}\ 
Suppose $\ell\geq 2n$. 
\begin{enumerate}
    \item There exists an $\OO(n)$-bi-Lipschitz map $\phi':\R^{n\times \ell}\to \R^{n(2\ell-2n+1)}$.
    \item There exists an $\E(n)$-bi-Lipschitz map
    $\psi':\R^{n\times \ell}\to \R^{n(2\ell-2n-1)}$.
       \item There exists an $\U(n)$-bi-Lipschitz map $\phi'_\C:\C^{n\times \ell}\to \R^{4n(\ell-n)}$.
    \item There exists an $\FF(n)$-bi-Lipschitz map
    $\psi'_\C:\C^{n\times \ell}\to \R^m$ where $4n(\ell-n-1)$.    
\end{enumerate}
\end{theorem}

\subsection{Applications}
In many data applications, the data has may have a lot of symmetries. One can exploit these symmetries to develop more effective machine learning algorithms. For example, convolutional neural network exploit the translation symmetries of images, video or audio signals.

In biological shape analysis, landmarks are discrete anatomical points on  biological specimens. For example, if we consider $\ell$ landmarks, then for one specimen the landmarks are located at the points $a_1,a_2,\dots,a_\ell\in \R^3$
and for another specimen the correponding landmarks are located at $b_1,b_2,\dots,b_\ell\in \R^3$. The distance between two specimen is $d_{\E(3)}(A,B)$
where $A=[a_1\ a_2\ \cdots\ a_\ell]$ and $B=[b_1\ b_2\ \cdots\ b_\ell]$. (Instead of the group $\E(3)$ one may want to exclude mirror images, and work with the subgroup $\operatorname{SE}(3)=\operatorname{SO}(3)\ltimes \R^3\subset \OO(3)\ltimes\R^3=\E(3)$).

Another application is space navigation by tracking stars. For example, if a group of $\ell$ stars is observed by a camera, we know the direction of the stars. These directions are $\ell$ points $a_1,a_2,\dots,a_\ell$ on the unit sphere in $\R^3$. To compare two $\ell$-tuples of stars
$A=[a_1\ a_2\ \cdots\ a_\ell]$ and $B=[b_1\ b_2\ \cdots\ b_\ell]$ we can consider  the distance $d_{\OO(3)}(A,B)$
or $d_{\operatorname{SO}(3)}(A,B)$. Since the stars are typically unlabeled, we also may want to consider
the action of the larger groups $\operatorname{SO}(3)\times S_{\ell}$ or $\OO(3)\times S_\ell$ where $S_\ell$
acts by permuting the $\ell$ stars.

For an action of $G$ on $\R^n$ there is a major advantage to having a $G$-bi-Lipschitz feature map. For example, consider the {\em nearest orbit problem}: Given a point $a\in \R^n$ and a database of points $b_1,b_2,\dots,b_k$ where $k$ is very large. Find $i$ for which the orbit $G\cdot b_i$ is nearest to the orbit $G\cdot a$. A simple linear search
computes
$d(G\cdot a,G\cdot b_j)=d_G(a,b_j)$ for $j=1,2,\dots,k$
and the running time will be $\Omega(k)$.

Suppose we have a $G$-bi-Lipschitz map $\phi:\R^n\to \R^m$ 
with
$$
d_G(x,y)\leq d(\phi(x),\phi(y))\leq Cd_G(x,y)
$$
for all $x,y\in \R^n$ and some fixed constant $C$. Using the $G$-bi-Lipschitz map, we can relate the nearest orbit problem to the nearest neighbor search problem in euclidean space. 
Given a point $u$ and a database of points $v_1,v_2,\dots,v_k\in \R^m$, 
the nearest neighbor problem in euclidean space seeks the index $i$ for which $d(u,v_i)$ is minimal. If $m=1$, the database
can be prepared such that $v_1,v_2,\dots,v_k$ are arranged from small to large. Then the index $i$ with $d(u,v_i)$ minimal can be found using a binary search. This algorithm runs in time $O(\log(k))$. If $m>1$ then there is a nearest neighbor search algorithm in $\R^m$ using an $m$-dimensional tree that runs (in the worst case) in time $O(mk^{1-1/m})$ (see~\cite{LW1977}). For $E\geq 1$, an $E$-approximate nearest neighbor an index $i$ such that $d(u,v_i)\leq Ed(u,v_j)$
for all $j$. The approximate nearest neighbor problem is easier than the exact nearest neighbor problem. 
For fixed $E>1$ and $m$, there exists an algorithm
for the $E$-approximate nearest neighbor search that runs in time $O(\log(k))$ (see~\cite{AMNSW1998}).

Using a nearest neighbor algorithm for euclidean space
we can find $i$ such that $\phi(b_i)$ is an $E$-approximate
nearest neighbor of $\phi(a)$. Then for 
all $j$ we have 
$d_G(a,b_i)\leq d(\phi(a),\phi(b_i))\leq E d(\phi(a),\phi(b_j))\leq EC d_G(a,b_j)$.
So the orbit $G\cdot b_i$ is the $(EC)$-approximate nearest norbit to $G\cdot a$.

\section{Orbit metrics}
\subsection{A metric on the orbit space}
For the action of a group $G$ on the metric space $\R^n$
there is a pseudo-metric on $\R^n/G$ defined as follows:
the distance between two orbits $G\cdot a$ and $G\cdot b$
is the infimum over all
$$
d(p_1,q_1)+d(p_2,q_2)+\cdots+d(p_r,q_r)
$$
where $r$ is a positive integer, $G\cdot p_1=G\cdot a$, $G\cdot q_r=G\cdot b$
and $G\cdot q_i=G\cdot p_{i+1}$ for $i=1,2,\dots,r-1$.

We will consider the special case where $G$ is a subgroup of the euclidean group $\E(n)$. It turns out that the infimum is actually a minimum, and this minimum is already attained for $r=1$. In this case, the pseudo-metric is an actual metric.

\begin{lemma}\label{lem:minexists}
Suppose that $G$ is a closed subgroup of the Lie group $\E(n)$. 
Then for every $a,b\in \R^n$, there exists an element $g\in G$ for which $d(g\cdot a,b)$ is minimal.
\end{lemma}
\begin{proof}
Consider the function $\theta:\E(n)\to [0,\infty)$ defined by
$\theta(g)=d(g\cdot a,b)$.
The function $\theta$ is continuous. Suppose that $C\geq 0$ is a constant
and let $S=\theta^{-1}([0,C])=\{g\in \E(n)\mid d(g\cdot a,b)\leq C\}$. 
If $g=(P,q)\in S$, then we have $d(q,b)\leq d(q,Pa+q)+d(Pa+q,b)\leq \|Pa\|+C=\|a\|+C$. This shows that $S$ is bounded. The inverse images of
bounded sets are bounded. So the map $\theta$ is proper.
It follows that $\theta$ is a closed map.
The subset
$D=\{d(g\cdot a,b)\mid g\in G\}\subseteq [0,\infty)$ is closed. Therefore this set has a smallest element.
\end{proof}

\begin{lemma}\label{lem:minequal}
    Suppose that $G$ is a closed subgroup of the Lie group $\E(n)$. If $a,b\in \R^n$ then we have
    $$
    \min_{g\in G}d(g\cdot a,b)=
    \min_{g,h\in G}d(g\cdot a,h\cdot b)=\min_{h\in G} d(a,h\cdot b).
    $$
\end{lemma}
\begin{proof}
Because $g^{-1}$ and $h^{-1}$ acts by  euclidean transformations, we have
$$d(h^{-1}g\cdot a,b)=(g\cdot a,h\cdot b)=d( a,g^{-1}h\cdot b).$$ 
Now the lemma follows from taking the minimum over all $g,h\in G$. 
\end{proof}
\begin{definition}
    Suppose that $G$ is a closed subgroup of $\E(n)$. If $a,b\in \R^n$ then We define the distance between the orbits $G\cdot a$ and $G\cdot b$ by $d(G\cdot a,G\cdot b)=\min_{g\in G,h\in H} d(g\cdot a,h\cdot b)$.
\end{definition}
By Lemma~\ref{lem:minequal}, the minimum exists and is well-defined. 

\begin{proposition}\label{prop:ismetric}
    The distance between orbits gives a metric on the orbit space $\R^n/G$. 
    \end{proposition}
\begin{proof}
Clearly $d(G\cdot a,G\cdot b)\geq 0$. If 
$d(G\cdot a,G\cdot b)=0$ then we have
$d(g\cdot a,h\cdot b)=0$ for some $g,h\in G$, so $g\cdot a=h\cdot b$ and  $G\cdot a=G\cdot b$. The distance $d$ on $\R^n/G$ is clearly symmetric. We also have the triangle inequality:
$$
d(G\cdot a,G\cdot c)=\min_{g,h\in G} d(g\cdot a,h\cdot c)\leq \min_{g,h\in G}\Big(d(g\cdot a,b)+d(b,h\cdot c)\Big)=d(G\cdot a,G\cdot b)+d(G\cdot b,G\cdot c).
$$

\end{proof}

\subsection{The pseudometrics $d_{\U(n)}$ and $d_{\OO(n)}$}\label{sec:2.3}
The group $\U(n)$ acts on $\C^{n\times \ell}$ by left multiplication.
We first study the pseudometric $d_{\U(n)}$. 
Every matrix $A\in \C^{n\times \ell}$ has a singular value decomposition $A=UDV^\star$ where $U\in \U(n)$, $V\in \U(\ell)$
and $D\in \C^{n\times \ell}$ is a matrix with nonnegative real entries
$\lambda_1\geq \lambda_2\geq \cdots\geq \lambda_{r}\geq 0$ ($r=\min\{\ell,n\}$) on the diagonal and zeroes elsewhere.
The numbers $\lambda_1,\lambda_2,\dots,\lambda_r$ are the singular values of $A$. The nuclear norm $\|A\|_\star$ of $A$ is the sum of the singular values $\|A\|_\star=\lambda_1+\lambda_2+\cdots+\lambda_r$.
We have $\real(\trace(A))\leq \|A\|_\star$ for all $A\in \C^{n\times \ell}$.

Finding the distance
$d_{\OO(n)}(A,B)$ between $A,B\in \R^{n\times \ell}$
is known as the orthogonal Procrustus problem and was solved in~\cite{Schonemann1966}. We will give a solution here and also to the unitary version of it.

\begin{lemma}\label{lem:dU}
For $A,B\in \C^{n\times \ell}$ we have
$$
d_{\U(n)}(A,B)^2=\|A\|^2+\|B\|^2-2\|AB^\star\|_\star.
$$
There exists a unitary matrix $W\in \U(n)$ such that
$WAB^\star$ is positive definite hermitian. For any such $W$ we have
$d_{\U(n)}(A,B)=\|WA-B\|$.
\end{lemma}
\begin{proof}
If $W\in \U(n)$ then we have 
\begin{multline}\label{eq:WAB}
\|WA-B\|^2=\trace((WA-B)(WA-B)^\star)
=\trace(WAA^\star W^\star+BB^\star
-WAB^\star-BA^\star W^\star)
=\\
\trace(AA^\star)+\trace(BB^\star)-\trace(WAB^\star)-\overline{\trace(WAB^\star)}=\\
=\|A\|^2+\|B\|^2-2\real(\trace(WAB^\star))\geq
\|A\|^2+\|B\|^2-2\|WAB^\star\|_\star=
\|A\|^2+\|B^2\|^2-2\|AB^\star\|_\star.
\end{multline}
Let $AB^\star=UDV^\star$ be the singular value decomposition where
$U,V\in \C^{n\times n}$ are unitary, and $D\in \C^{n\times n}$ is diagonal with nonnegative real entries. For $W=VU^\star$ we have
that $WAB^\star=VU^\star UDV^\star=VDV^\star$ is nonnegative semidefinite hermitian
and 
$$\trace(WAB^\star)=\trace(VDV^\star)=\trace(D)=\|AB^\star\|_\star.$$
We have equality in~(\ref{eq:WAB}), so $\|WA-B\|=\min_{g\in \U(k)}\|gA-B\|=d_{\U(k)}(A,B)$
and $d_{\U(k)}(A,B)^2=\|WA-B\|^2=\|A\|^2+\|B\|^2-2\|AB^\star\|_\star$.
\end{proof}

\begin{corollary}\label{cor:dOisdU}
If $A,B\in \R^{n\times \ell}$ then there exists an orthogonal matrix $W\in \OO(n)$
such that $WAB^\tr$ is a nonnegative definite symmetric matrix and
$d_{\OO(n)}(A,B)=\|WA-B\|=d_{\U(n)}(A,B)$.
\end{corollary}
\begin{proof}
In the proof of Lemma~\ref{lem:dU} we can take the singular value decomposition $AB^\star=UDV^\star=UDV^\tr$
such that $U$ and $V$ are real orthogonal. Then $W=UV^\tr\in \R^{n\times n}$ is real and orthogonal. So we have $d_{\U(n)}(A,B)=\min_{Z\in \U(n)}\|ZA-B\|=\min_{Z\in \OO(n)}\|ZA-B\|=d_{\OO(n)}(A,B)$.
\end{proof}

\subsection{The pseudometrics $d_{\E(n)}$ and $d_{\FF(n)}$}

The group $\E(n)=\OO(n)\ltimes \R^n$ is a semi-direct product of the orthogonal group $\OO(n)$ and the group $(\R^n,+)$ of translations.
We also have a semi-direct product $\FF(n)=\U(n)\ltimes \C^n$.
We define $\pi_\C:\C^{n\times \ell}\to \C^{n\times \ell}$ by
$$
\pi_\C\left(\begin{bmatrix}
    a_1 & a_2 & \cdots & a_\ell\end{bmatrix}\right)=\begin{bmatrix}
        a_1-\overline{a},a_2-\overline{a},\cdots,a_{\ell}-\overline{a}
        \end{bmatrix},
$$
where $\overline{a}=\frac{1}{\ell}(a_1+a_2+\cdots+a_\ell)$ is the average. We let $\pi:\R^{n\times \ell}\to \R^{n\times \ell}$ be the restriction of $\pi_\C$.
\begin{lemma}\label{lem:9}
    For all $A,B\in \C^{n\times \ell}$ we have
    $d_{\C^n}(A,B)=d(\pi_\C(A),\pi_\C(B))$. 
    Also, for all $A,B\in \R^{n\times \ell}$ we have
    $d_{\R^n}(A,B)=d(\pi(A),\pi(B))$. 
\end{lemma}
\begin{proof}
For $A=[a_1\ a_2\ \cdots\ a_\ell]$ and $B=[b_1\ b_2\ \cdots\ b_\ell]$ in $\C^{n\times \ell}$, we have
\begin{equation}\label{eq:2}
    d_{\C^n}(A,B)^2=\min_{z\in \C^n}\sum_{i=1}^\ell \|z+a_i-b_i\|^2.
\end{equation}
If we take $p_i=z+a_i-b_i$ in (\ref{eq:2}), then $\overline{p}=z+\overline{a}-\overline{b}$ and 
we get
\begin{multline*}
d_{\C^n}(A,B)^2=\min_{z\in \C^n}\sum_{i=1}^\ell \|p_i\|^2=\min_{z\in \C^n}
\ell \|\overline{p}\|^2+\sum_{i=1}^\ell \|p_i-\overline{p}\|^2=
\min_{z\in \C^n}\ell\|z+\overline{a}-\overline{b}\|^2+ \sum_{i=1}^\ell \|(a_i-\overline{a})-(b_i-\overline{b})\|^2=\\= 
\sum_{i=1}^\ell \|(a_i-\overline{a})-(b_i-\overline{b})\|^2=d(\pi_\C(A),\pi_\C(B))^2.
\end{multline*}
If $A,B\in \R^{n\times \ell}$ then we may view $A$ and $B$ as elements of $\C^{n\times \ell}$. We get
$$
d_{\R^n}(A,B)\leq d(\pi(A),\pi(B))=
d(\pi_\C(A),\pi_\C(B))=d_{\C^n}(A,B)\leq d_{\R^n}(A,B),
$$
so all the inequalities are equalities.
\end{proof}
\begin{lemma}\label{lem:10}
For $A,B\in \C^{n\times \ell}$ we have $d_{\FF(n)}(A,B)=d_{\U(n)}(\pi_\C(A),\pi_\C(B))$. Also, if $A,B\in \R^{n\times \ell}$ we have $d_{\E(n)}(A,B)=d_{\OO(n)}(\pi(A),\pi(B))$.
\end{lemma}
\begin{proof}
We have 
$$
d_{\FF(n)}(A,B)=d_{\FF(n)}(\pi_\C(A),\pi_\C(B))\leq
d_{\U(n)}(\pi_\C(A),\pi_\C(B)).
$$
There exists an element $g=(P,q)\in \FF(n)$ with
$d_{\FF(n)}(A,B)=d(g\cdot A,B)$.
Note that 
$$\pi_\C(g\cdot A)=\pi_\C(q\cdot (P\cdot A))=\pi_\C(P\cdot A)=P\cdot \pi_\C(A).
$$
So we have
$$
d_{\U(n)}(\pi_\C(A),\pi_\C(B))\leq d(P\cdot \pi_\C(A),\pi_\C(B))=d(\pi_\C(g\cdot A),\pi_\C(B))\leq
d(g\cdot A,B)=d_{\FF(n)}(A,B).
$$
This proves that $d_{\U(n)}(\pi_\C(A),\pi_\C(B))=d_{\FF(n)}(A,B)$.

Suppose $A,B\in \R^{n\times \ell}$. Note again that $\pi(A)$ and $A$ lie in the same $\E(n)$-orbit, and the same holds for $\pi(B)$ and $B$, so $d_{\E(n)}(A,B)=d_{\E(n)}(\pi(A),\pi(B))$. 
By Corollary~\ref{cor:dOisdU} and Lemma~\ref{lem:9} we get 
\begin{multline*}
 d_{\OO(n)}(\pi(A),\pi(B))=
d_{\U(n)}(\pi_\C(A),\pi_\C(B))=\\=d_{\FF(n)}(A,B)\leq d_{\E(n)}(A,B)=d_{\E(n)}(\pi(A),\pi(B))\leq d_{\OO(n)}(\pi(A),\pi(B)).
\end{multline*}
All the inequalities are equalities, so $d_{\E(n)}(A,B)=d_{\OO(n)}(\pi(A),\pi(B))$.

\end{proof}

\section{Proofs of the main results}

\subsection{Unitary and orthogonal group actions}
For a  positive semi-definite hermitian matrix $B\in \C^{\ell\times \ell}$, let $\sqrt{B}$
be the unique positive semi-definite matrix whose square is $B$. We define
$\phi_\C:\C^{n\times \ell}\to \C^{\ell\times \ell}$ by $\phi_\C(A)=\sqrt{A^\star A}$.
The image is contained in the set of positive semi-definite hermitian matrices. Let $\phi:\R^{n\times \ell}\to \R^{\ell\times \ell}$ be the restriction of $\phi_\C$.
\begin{proof}[Proof of Theorem~\ref{theo:5}]

% \begin{proposition}\label{prop:UbiLipschitz}
% The map $\phi_\C$ is $\U(k)$-bi-Lipschitz. For all $A,B\in \C^{k\times \ell}$ we have
% $$
%  d_{\U(k)}(A,B)\leq \|\phi_\C(A)-\phi_\C(B)\|\leq \sqrt{2}d_{\U(k)}(A,B).
% $$
% \end{proposition}
% \begin{proof}
Suppose that $A,B\in \C^{n\times \ell}$.
Let $A=U_1D_1V_1^\star$, $B=U_2D_2V_2^\star$
be the singular value decompositions,
where $U_i\in \C^{n\times r}$, $V_i\in \C^{n\times r}$,
$r=\min\{n,\ell\}$ such that
$U_i^\star U_i={\bf 1}_n$
and $V_i^\star V_i={\bf 1}_\ell$ for $i=1,2$.
We have
\begin{multline*}
2\|A-B\|^2-\|\sqrt{A^\star A}-\sqrt{B^\star B}\|^2-
\|\sqrt{AA^\star}-\sqrt{BB^\star}\|^2=2\trace\big(A^\star A+B^\star B-A^\star B-B^\star A)-\\
\trace(A^\star A+B^\star B-2\sqrt{A^\star A}\sqrt{B^\star B})-
\trace(AA^\star+BB^\star-2\sqrt{AA^\star}\sqrt{BB^\star})=\\
=2\big(\trace\big(\sqrt{A^\star A}\sqrt{B^\star B})+\trace(\sqrt{AA^\star}\sqrt{BB^\star})-\trace(A^\star B+B^\star A)\big)=\\
=2\big(\trace(V_1D_1V_1^\star V_2D_2V_2^\star)+
\trace(U_1D_1U_1^\star U_2D_2U_2^\star)-
\trace(V_1D_1U_1^\star U_2D_2V_2^\star+V_2D_2U_2^\star U_1D_1V_1^\star)\Big)=\\
2\trace\big(V_2^\star V_1D_1U_1^\star U_2D_2+U_2^\star U_1D_1U_1^\star U_2D_2-V_2^\star V_1 D_1U_1^\star U_2D_2-U_2^\star U_1D_1V_1^\star V_2D_2\big)=\\
2\trace\big((V_2^\star V_1-U_2^\star U_1)D_1(V_1^\star V_2-U_1^\star U_2)D_2\big)=2\trace(C^\star D_1 CD_2)\geq 0,
\end{multline*}
where $C=V_1^\star V_2-U_1^\star U_2$. To see the last inequality,
note that $C^\star D_1 C$ and $D_2$ are positive semi-definite hermitian matrices and that the trace of a product of two positive semi-definite matrices is nonnegative. It now follows that
$$
2\|A-B\|^2-\|\phi_\C(A)-\phi_\C(B)\|^2=2\|A-B\|^2-\|\sqrt{A^\star A}-\sqrt{B^\star B}\|^2\geq \|\sqrt{AA^\star}-\sqrt{BB^\star}\|^2\geq 0
$$
and 
$$
\|\phi_\C(A)-\phi_\C(B)\|\leq \sqrt{2}\|A-B\|.
$$
There exists a unitary matrix $W\in \U(n)$
with $d_{\U(n)}(A,B)=\|WA-B\|$. Then we get
$$
\|\phi_\C(A)-\phi_\C(B)\|=\|\phi_\C(WA)-\phi_\C(B)\|
\leq \sqrt{2}\|WA-B\|=\sqrt{2}d_{\U(n)}(A,B).
$$
Define $n\times n$ matrices $W_1,W_2$ as follows:
\begin{enumerate}
\item If $n\leq \ell$, then $r=n$. We extend $V_1,V_2\in \C^{\ell\times n}$
to unitary $\ell\times \ell$ matrices $[V_1\ \widetilde{V}_1]$ and $[V_2\ \widetilde{V}_2]$ respectively.
Define $W_1=U_2^\star (V_1V_2^\star+\widetilde{V}_1\widetilde{V}_2^\star)U_1$
and $W_2=U_2^\star (V_1V_2^\star-\widetilde{V}_1\widetilde{V}_2^\star)U_1$.
\item If $n\geq \ell$, then $r=\ell$. We extend $U_1,U_2\in \C^{n\times \ell}$
to unitary $n\times n$ matrices $[U_1\ \widetilde{U}_1]$ and $[U_2\ \widetilde{U}_2]$ respectively.
Define $W_1=U_2^\star V_1V_2^\star U_1+\widetilde{U}_2^\star V_1V_2^\star \widetilde{U}_1$
and $W_2=U_2^\star V_1V_2^\star U_1-\widetilde{U}_2^\star V_1V_2^\star \widetilde{U}_1$.
\end{enumerate}
In both cases, $W_1$ and $W_2$ are unitary, and $W_1+W_2=2U_2^\star V_1V_2^\star U_1$.
We have
\begin{multline*}
\trace(W_1AB^\star)+\trace(W_2AB^\star)=\trace((W_1+W_2)AB^\star)=
2\trace((U_2V_2^\star V_1 U_1^\star) AB^\star)=\\
=2\trace\big((U_2V_2^\star V_1 U_1^\star)(U_1D_1V_1^\star V_2D_2U_2^\star)\big)=2\trace( V_1 (U_1^\star U_1)D_1 V_1^\star V_2 D_2 (U_2^\star U_2)V_2^\star)=\\2\trace\big((V_1^\star D_1 V_1^\star)( V_2 D_2 V_2^\star)\big)=2\trace\big(\sqrt{A^\star A}\sqrt{B^\star B}\big).
\end{multline*}
For some choice $i\in \{1,2\}$ we have
$$
\real(\trace(W_iAB^\star))\geq \trace\big(\sqrt{A^\star A}\sqrt{B^\star B}\big).
$$
We get
\begin{multline*}
\|W_iA-B\|^2=\trace(AA^\star+BB^\star)-2\real(\trace(W_iAB^\star))\leq\\
\trace(AA^\star+BB^\star)-2\trace(\sqrt{A^\star A}\sqrt{B^\star B})=
\big\|\sqrt{A^\star A}-\sqrt{B^\star B}\big\|^2.
\end{multline*}
This shows that $\|W_iA-B\|\leq \|\phi_\C(A)-\phi_\C(B)\|$
and therefore $d_{\U(n)}(A,B)\leq \|\phi_\C(A)-\phi_\C(B)\|$.

The image of $\phi_\C$
is contained in the space $H^2(\C^\ell)$ of hermitian matrices of size $\ell\times \ell$. This space has real dimension $\ell^2$. By choosing an orthonormal basis of $H^2(\C^\ell)$ we can view $\phi_\C$ as a map from $\C^{n\times \ell}\to \R^{\ell^2}$.
\end{proof}

\begin{proof}[Proof of Theorem~\ref{theo:3}]
% \begin{corollary}
% The map $\phi_{\R}$ is $\OO(k)$-bi-Lipschitz. For all $A,B\in \R^{k\times \ell}$ we have
% $$
%  d_{\OO(k)}(A,B)\leq \|\phi_\R(A)-\phi_\R(B)\|\leq \sqrt{2}d_{\OO(k)}(A,B).
% $$
% \end{corollary}
% \begin{proof}
From Theorem~\ref{theo:5} and Corollary~\ref{cor:dOisdU} follows that
$$
d_{\OO(n)}(A,B)=d_{\U(n)}(A,B)\leq \|\phi_\C(A)-\phi_\C(B)\|=\|\phi(A)-\phi(B)\|\leq \sqrt{2}d_{\U(n)}(A,B)=
 d_{\OO(n)}(A,B).
$$
We make the simple observation that the image of $\phi$ is contained in the space $S^2(\R^\ell)$ of symmetric $\ell\times \ell$ matrices. This space has dimension ${\ell+1\choose 2}$. By choosing an orthonormal basis of $S^2(\R^\ell)$ we can view $\phi$ as a function from $\R^{n\times \ell}$ to $\R^{\ell(\ell+1)/2}$.
\end{proof}

\subsection{Euclidean group actions}

\begin{proof}[Proof of Theorem~\ref{theo:4}]
Suppose that $A,B\in \R^{n\times \ell}$.
Using Theorem~\ref{theo:3} and Lemma~\ref{lem:10} we have
\begin{multline*}
d_{\E(n)}(A,B)=d_{\OO(n)}(\pi(A),\pi(B))\leq
d(\phi\circ \pi(A),\phi\circ \pi (B))=d(\psi(A),\psi(B))\leq \\ 
\sqrt{2} d_{\OO(n)}(\pi(A),\pi(B))=\sqrt{2} d_{\E(n)}(A,B).
\end{multline*}
Let ${\bf 1}=[1\ 1\ \cdots\ 1]^\tr$.
Suppose $a_1,a_2,\dots,a_\ell\in \R^n$ and let $A=[a_1\ a_2\ \cdots\ a_\ell]$. 
We have
$\psi(A)=\sqrt{B^\tr B}$,
where
$$
B=\pi(A)=\begin{bmatrix} a_1-\overline{a},a_2-\overline{a},\dots,a_\ell-\overline{a}\end{bmatrix}.
$$
We have $B{\bf 1}=\sum_{i=1}^\ell (a_i-\overline{a})=0$. This implies that $\psi(A){\bf 1}=0$ and ${\bf 1}^\tr \psi(A)=0$. We choose a orthonormal basis
$w_1,w_2,\dots,w_n$ with $w_n=\frac{1}{n}{\bf 1}$
and we set $W=[w_1\ w_2\ \cdots\ w_n]$.
Then $W$ is an orthogonal matrix.
Now $W\psi(A)W^\tr$ is of the form
$$
\begin{bmatrix}
    \widetilde{\psi}(A) & \\ & 0
\end{bmatrix}
$$
where $\widetilde{\psi}(A)$ is a symmetric $(\ell-1)\times (\ell-1)$ matrix. By choosing an orthonormal basis of the space of symmetric $(\ell-1)\times (\ell-1)$ matrices, we can view $\psi$
as a map from $\R^{n\times \ell}$ to $\R^{\ell(\ell-2)/2}$.
\end{proof}
\begin{proof}[Proof of Theorem~\ref{theo:6}]
Suppose, $A,B\in \C^{n\times \ell}$.
Using Theorem~\ref{theo:5} and Lemma~\ref{lem:10} we have
\begin{multline*}
d_{\FF(n)}(A,B)=d_{\U(n)}(\pi_\C(A),\pi_\C(B))\leq
d(\phi_\C\circ \pi_\C(A),\phi_\C\circ \pi_\C(B))=d(\psi_\C(A),\psi_\C(B))\leq \\
\sqrt{2} d_{\U(n)}(\pi_\C(A),\pi_\C(B))=\sqrt{2} d_{\FF(n)}(A,B).
\end{multline*}
After a base change in $\C^{\ell\times \ell}$ as in the proof of Theorem~\ref{theo:5}, we get that $\psi_\C(A)$
is of the form
$$
\begin{bmatrix}
    \widetilde{\psi}_\C(A) & \\ & 0
\end{bmatrix}
$$
where $\widetilde{\psi}_\C(A)$ is a hermitian $(\ell-1)\times (\ell-1)$ matrix. So the image of $\psi_\C$ is
contained in a subspace of real dimension $(\ell-1)^2$.
\end{proof}
\section{Dimension reduction}

\subsection{Bi-Lipschitz embeddings of matrices of bounded rank}
Let $X_r\subseteq \C^{\ell\times \ell}$ be set of matrices of rank~$\leq r$. This is an algebraic variety of dimension $r(2\ell-r)$ and codimension $\ell^2-r(2\ell-r)=(\ell-r)^2$.
There exists a subspace $W_r$ of dimension 
$(\ell-r)^2$ with $W_r\cap X_r=\{0\}$. In fact, any subspace $W_r$ of dimension $(\ell-r)^2$ in general position will do. We will discribe an explicit choice of $W_r$ later.
\begin{lemma}
    There exists a constant $D<1$ with
    $|\langle A,B\rangle|\leq D$ for all $A\in X_r$ and $B\in W_r$.
\end{lemma}
\begin{proof}
The function $|\langle A, B\rangle|$ on the set
 compact set
$$
\{(A,B)\mid A\in W_r,B\in X_r,\|A\|=\|B\|=1\}
$$
has a maximum $D$. Take $A\in W_r$ and $B\in X_r$ with $\|A\|=\|B\|=1$
with $D=|\langle A,B\rangle|$. By Cauchy-Schwarz, $D\leq \|A\|\|B\|=1$.
Suppose that $D=1$. Then $A=\pm B\in W_r\cap X_r$. This is a contradiction, so $D<1$.
\end{proof}
Set $C=\sqrt{1-D^2}>0$.
Let $p_r: X_r\to W_r^\perp$ be the projection onto
the orthogonal complement 
$$
W_r^\perp=\{B\in \C^{\ell\times \ell}\mid \forall A\in W_r\ \langle A,B\rangle=0\}.
$$
\begin{lemma}
    For any matrix $A\in X_r$ we have
    $$
    C\|A\|\leq \|\rho(A)\|\leq \|A\|.
    $$
\end{lemma}
\begin{proof}
 Because $p_r$ is a linear projection, it suffices to prove the Lemma in the case where $\|A\|=1$.
   Write $A=A_1+A_2$ with $A_1\in W_r$ and $A_2\in W_r^\perp$. Then we have 
   $$
   \|A_1\|^2=\langle A_1,A_1\rangle=\langle A_1,A_1+A_2\rangle=\langle A_1,A\rangle\leq C\|A_1\|,
   $$
   so $\|A_1\|\leq C$ and $\|A_2\|=\sqrt{1-\|A_1\|^2}\geq \sqrt{1-D^2}=C$.   
\end{proof}
Let $\C[x]_{<\ell}$ be the $\ell$-dimensional vector space of polynomials of degree $<\ell$.
By choosing a basis of $\C[x]_{<\ell}$ we can
identify the vector spaces $\C[x]_{<\ell}$ and $\C^{d+1}$
 We can also make identifications between the vector spaces
$$
\C[x,y]_{<\ell,<\ell}\cong \C[x]_{< \ell}\otimes \C[x]_{<\ell}\cong \C^{\ell}\otimes \C^{\ell}\cong \C^{\ell\times \ell}.
$$
Here $\C[x,y]_{< \ell,<m}$ is the space of polynomials in $x$ and $y$
that have degree $<\ell$ in the $x$-variable and degree $< m$ in the $y$-variable.
The rank of an element $A(x,y)\in \C[x,y]_{<\ell,<m}$ is the smallest nonnegative integer for which we can write $A(x,y)=\sum_{i=1}^r f_i(x)g_i(y)$.
We have two partial derivatives $\partial_x:\C[x,y]_{<\ell,<m}\to \C[x,y]_{<\ell-1,<m}$
and $\partial_y:\C[x,y]_{<\ell,<m}\to \C[x,y]_{<\ell,<m-1}$.
Setting $y=x$ gives a linear map $\Theta:\C[x,y]_{<\ell,<m}\to \C[x]_{<\ell+m-1}$.

\begin{definition}
    Define the subspace $W_{r}\subseteq \C[x,y]_{<\ell,<\ell}$ as the set of all elements that lie in the ideal generated by $(x-y)^r$. 
\end{definition}

\begin{lemma}
    We have
    \begin{enumerate}
        \item $\dim W_r=(\ell-r)^2$;
        \item  if $A\in W_r\subseteq \C^{\ell\times \ell}$ then $A^\tr\in W_r$.
        \item if $A\in W_r$ then $\overline{A}\in W_r$ where $\overline{A}$ is the complex conjugate of $A$.
    \end{enumerate}
\end{lemma}
\begin{proof}\ 
h
(1) The space $W_r$ has a basis given by the polynomials
     $(x-y)^r x^iy^j$ with $0<i,j<\ell-r$. Since there are $\ell-r$ choices for $i$ and $j$, $\dim W_r=(\ell-r)^2$.

    (2) Mapping $A$ to $A^\tr$ in $\C^{\ell\times \ell}$ corresponds to sending a polynomial $f(x,y)$ to $f(y,x)$ in $\C[x,y]_{<\ell,<ell}$
    and it is clear that $f(x,y)\in W_r$ if and only if $f(y,x)\in W_r$.

    (3) This is clear because $W_r$ has a basis of real polynomials. 
    \end{proof}

\begin{proposition}
    We have $W_r\cap X_r=\{0\}$.
    \end{proposition}
    \begin{proof}

    Suppose $h(x,y)\in W_r$ has rank $\leq r$.
    Then we can write $h(x,y)=\sum_{i=1}^r f_i(x)g_i(y)$
    where $f_1(x),f_2(x),\dots,f_r(x)$ are linearly independent. For $k<r$ the polynomial is 
    $\partial_x^k h(x,y)=\sum_{i=1}^r f_i(x)^{(k)}g_i(y)$
    divisible by  $(y-x)^{r-k}$ and hence by $y-x$.
    Setting $x$ equal to $y$ gives 
   $ \sum_{i=1}^r f_i(y)^{(k)}g_i(x)$ for $k=0,1,\dots,r-1$.
    We can rewrite this as
\begin{equation}\label{eq:Wronski}
    W(x)\begin{bmatrix}
        g_1(x)\\ g_2(x)\\ \vdots \\ g_r(x)
    \end{bmatrix}=\begin{bmatrix}0 \\ 0 \\ \vdots \\ 0\end{bmatrix},
    \end{equation}
    where $W(x)$ is the Wronski matrix
    $$
    \begin{bmatrix}
        f_1(x) & f_2(x) & \cdots & f_r(x)\\
        f_1'(x) & f_2'(x) & \cdots & f_r'(x)\\
        \vdots & \vdots & & \vdots\\
        f_1^{(r-1)}(x) & f_2^{(r-1)}(x) & \cdots & f_r^{(r-1)}(x)
    \end{bmatrix}
    $$
    Since $f_1(x),f_2(x),\dots,f_r(x)$ are linearly independent, the Wronskian $\det(W(x))$ is a nonzero function. So $W(x)$ is invertible, 
    $g_1(x)=g_2(x)=\cdots=g_r(x)=0$ and $h(x,y)=0$.
    \end{proof}

\subsection{Dimension reduction}
Let $W_r\subseteq \C^{\ell\times \ell}$ and $\rho_r:\C^{\ell\times \ell}\to W_r^\perp$ as in the previous section.
\begin{lemma}\label{lem:dim_rho}
Suppose that $\ell\geq 2n>0$.
\begin{enumerate}
    \item The real vector space $\rho_{2n}(S^2(\R^\ell))$ has dimension
    $n(2\ell-2n+1)$.
    \item The real vector space
    $\rho_{2n}(H^2(\C^\ell))$ has dimension
    $4n(\ell-n)$.
    
\end{enumerate}    
\end{lemma}
\begin{proof}\ 

(1) The space $\rho_{2n}(S^2(\R^\ell))$ is isomorphic
to $S^2(\R^\ell)/(W_{2n}\cap S^2(\R^\ell))$.
A basis of the space $W_{2n}\cap S^2(\R^\ell))$
is $(x-y)^{2n}(x+y)^j(xy)^k$ with $j+k<\ell-2n$,
so this space has dimension ${\ell+1-2n\choose 2}$.
So  the dimension of $S^2(\R^\ell)/(W_{2n}\cap S^2(\R^\ell))$ is equal to ${\ell+1\choose 2}-{\ell+1-2n\choose 2}=n(2\ell-2n+1)$.

(2)  The space $\rho_{2n}(H^2(\C^\ell))$ is isomorphic
to $H^2(\C^\ell)/(W_{2n}\cap H^2(\C^\ell))$.
A basis of the real vector space $W_{2n}\cap H^2(\C^\ell))$
is $(x-y)^{2n}(x+y)^j(x-y)^ki^k$ with $0\leq j,k<\ell-2n$,
so this space has dimension $(\ell-2n)^2$. So the dimension of $H^2(\C^\ell)/(W_{2n}\cap H^2(\C^\ell))$
is $\ell^2-(\ell-2n)^2=4n\ell(\ell-n)$.

\end{proof}

\begin{proof}[Proof of Theorem~\ref{theo:bi-Lipschitz}] \ 

(1) Consider the composition $\phi'=\rho_{2n}\circ\phi$, where $\phi:\R^{n\times \ell}\to S^2(\R^\ell)\subseteq \R^{\ell\times \ell}\subseteq \C^{\ell\times \ell}$.
 For $A,B\in \R^{n\times \ell}$ we have
$$
\|\phi'(A)-\phi'(B)\|=\|\rho_{2n}\circ\phi(A)-\rho_{2n}\circ\phi(B)\|=\|\rho_{2n}(\phi(A)-\phi(B))\|\geq C\|\phi(A)-\phi(B)\|\geq Cd_{\OO(n)}(A,B),
$$
because $\phi(A),\phi(B)\in X_{n}$ and $\phi(A)-\phi(B)\in X_{2n}$. We also have
$$
\|\phi'(A)-\phi'(B)\|=\|\rho_{2n}(\phi(A)-\phi(B))\|\leq
\|\phi(A)-\phi(B)\|\leq \sqrt{2}d_{\OO(n)}(A,B).
$$
This proves that $\phi'$ is $\OO(n)$-bi-Lipschitz.

The image is contained in $\rho_{2n}(S^2(\R^\ell))$
which has dimension $n(2\ell-2n+1)$ by Lemma~\ref{lem:dim_rho}.

(2) We can view $\psi$ as a map from $\R^{n\times \ell}$
to $S^2(\R^{\ell-1})\subseteq \R^{(\ell-1)\times (\ell-1)}\subseteq \C^{(\ell-1)\times (\ell-1)}$. Define $\psi'=\rho_{2n}\circ \psi$. A similar argument as in part (1) shows that $\psi'$ is $\E(n)$-bi-Lipschitz.
The image of $\psi'$ is contained
in the space $\rho_{2n}(S^2(\C^{\ell-1}))$ which has dimension $n(2\ell-2n-1)$ by Lemma~\ref{lem:dim_rho}
(where we replace $\ell$ by $\ell-1$).

(3) Consider the composition $\phi'_\C=\rho_{2n}\circ \phi_\C$. A similar proof as in part (1) shows that $\psi'_\C$ is $\U(n)$-bi-Lipschitz. The image of $\phi'_\C$ is contained in $\rho_{2n}(H^2(\C^\ell))$
which has dimension $4n(\ell-n)$ by Lemma~\ref{lem:dim_rho} by Lemma~\ref{lem:dim_rho}.

(4) Consider the composition $\psi'_\C=\rho_{2n}\circ \psi_\C$, where $\psi:\C^{n\times \ell}\to \C^{(\ell-1)\times (\ell-1)}$. A similar proof as in part (1) shows that $\psi'_\C$ is $\F(n)$-bi-Lipschitz. The image of $\phi'_\C$ is contained in $\rho_{2n}(H^2(\C^{\ell-1}))$
which has dimension $4n(\ell-n-1)$ by Lemma~\ref{lem:dim_rho} (with $\ell$ replaced by $\ell-1$).
\end{proof}

\section{Euclidean triangles}

We consider the action of the euclidean group $\E(2)$ on the space $\R^{2\times 3}$ consisting of triples of points.  Define $\gamma:(\R^2)^3\to \R^3$ by
$$
\gamma(\begin{bmatrix} a_1 & a_2 & a_3\end{bmatrix})=(d(a_2,a_3),d(a_3,a_1),d(a_1,a_2)).
$$
From the triangle inquality, we see that the image of $\gamma$ is the set of all $(x_1,x_2,x_3)\in \R^3$ for which $x_1,x_2,x_3,x_1+x_2-x_3,x_2+x_3-x_1,x_3+x_1-x_2$ are nonnegative.
This is a cone spanned by the rays
through the vectors $(0,1,1)$, $(1,0,1)$, $(1,1,0)$. A triangle is, up to isometry,  determined by the lengths of the edges, so $\gamma$ is orbit separating. For $A=[a_1\ a_2\ a_3]$ and $B=[b_1\ b_2\ b_3]$ we have
\begin{multline}\label{eq:dgamma}
d(\gamma(A),\gamma(B))^2=
\big|\|a_1-a_2\|-\|b_1-b_2\|\big|^2+
\big|\|a_2-a_3\|-\|b_2-b_3\|\big|^2
+\big|\|a_3-a_1\|-\|b_3-b_1\|\big|^2\leq \\
\|(a_1-b_1)-(a_2-b_2)\|^2+
\|(a_2-b_2)-(a_3-b_3)\|^2+
\|(a_3-b_3)-(a_1-b_1)\|^2=\\
3\big(\|a_1-b_1\|^2+\|a_2-b_2\|^2+
\|a_3-b_2\|^2\big)-\|(a_1-b_1)+(a_2-b_2)+(a_3-b_3)\|^2\leq 3d(A,B)^2.
\end{multline}
If we choose $g\in \E(3)$ such that
$d(A,g\cdot B)=d_{\E(2)}(A,B)$
and replace $B$ by $g\cdot B$ in
(\ref{eq:dgamma}), we get
$$d(\gamma(A),\gamma(B))\leq \sqrt{3}d_{\E(2)}(A,B).
$$
However, $\gamma$ is not $\E(2)$-bi-Lipschitz as we will see.
Suppose that there exists a constant $C_1$ with
$$
C_1d_{\E(2)}(A,B)\leq d(\gamma(A),\gamma(B))
$$
for all $A,B\in \R^{2\times 3}$.
For example, let
$$
A=\begin{bmatrix}
    1 & 0 & -1\\
    0 & 0 & 0
\end{bmatrix}
\mbox{ and }
B=\begin{bmatrix}
    1 & 0 & -1\\
    -\varepsilon & 2\varepsilon & -\varepsilon
    \end{bmatrix}
$$
Let $\pi:\R^{n\times 3}\to \R^{n\times 3}$ as in Section~\ref{sec:2.3}. Then we have $\pi(A)=A$ and $\psi(B)=B$. The matrix
$$
AB^\tr=\begin{bmatrix}
    2 & 0 \\
    0 & 0
\end{bmatrix}
$$
is positive semi-definite and by Corollary~\ref{cor:dOisdU}
(with $W=I$) we get $d_{\OO(2)}(A,B)=d(A,B)$.
By Lemma~\ref{lem:10}, 
$$d_{\E(2)}(A,B)=d_{\OO(2)}(\pi(A),\pi(B))=d_{\OO(2)}(A,B)=d(A,B)=\sqrt{6}\varepsilon.
$$
On the other hand,
$$
d(\gamma(A),\gamma(B))=d((1,2,1),(\sqrt{1+9\varepsilon^2},2,\sqrt{1+9\varepsilon^2}))
$$
If $\varepsilon$ is small, then $\sqrt{1+9\varepsilon^2}=1+\frac{9}{2}\varepsilon^2+O(\varepsilon^4)$.
We have
$$
C_1\sqrt{6}\varepsilon\leq C_1 d_{\E(2)}(A,B)\leq
d(\gamma(A),\gamma(B))=\frac{9\sqrt{2}}{2}\varepsilon^2+O(\varepsilon^4).
$$
This gives a contradiction as $\varepsilon\downarrow 0$. 

In Theorem~\ref{theo:4} we constructed a map $\psi:\R^{2\times 3}\to \R^{3\times 3}$ with the property
$$
d_{\E(2)}(A,B)\leq d(\psi(A),\psi(B))\leq \sqrt{2}d_{\E(2)}(A,B).
$$
Consider the orthogonal matrix
$$
U=\begin{bmatrix}
    \frac{-1}{\sqrt{2}} & \frac{-1}{\sqrt{6}}& \frac{1}{\sqrt{3}}\\
    \frac{1}{\sqrt{2}}  & \frac{-1}{\sqrt{6}} & \frac{1}{\sqrt{3}}\\
    0 & \frac{2}{\sqrt{6}} & \frac{1}{\sqrt{3}}
\end{bmatrix}.
$$
$$
U=\begin{bmatrix} -\frac{1}{\sqrt{2}} & \frac{1}{\sqrt{2}} & 0\\
\frac{-1}{\sqrt{6}} & \frac{-1}{\sqrt{6}} & \frac{2}{\sqrt{6}}\end{bmatrix}
$$
Then we have
$$
\begin{bmatrix}
    a_1-\overline{a} & a_2-\overline{a} & a_3-\overline{a}
\end{bmatrix}=\widetilde{A}U^\tr,\mbox{ where }
\widetilde{A}=
\begin{bmatrix} 
\displaystyle \frac{a_2-a_1}{\sqrt{2}} & \displaystyle \frac{2a_3-a_1-a_2}{\sqrt{6}} \end{bmatrix} 
$$
Define $\Psi_1(A),\Psi_2(A),\Psi_3(A)$ by
$$
\begin{bmatrix}
    \Psi_1(A) & \frac{1}{\sqrt{2}}\Psi_3(A)\\
    \frac{1}{\sqrt{2}}\Psi_3(A) & \Psi_2(A)
\end{bmatrix}=(\widetilde{A}^\tr \widetilde{A})^{1/2}=
\begin{bmatrix}
\displaystyle \frac{\|a_2-a_1\|^2}{2} &
\displaystyle \frac{\langle a_2-a_1,2a_3-a_1-a_2\rangle}{2\sqrt{3}}  \\
\displaystyle \frac{\langle a_2-a_1,2a_3-a_1-a_2\rangle}{2\sqrt{3}} &
\displaystyle \frac{\|2a_3-a_1-a_2\|^2}{6}
\end{bmatrix}^{1/2}.
$$
It follows that
\begin{multline*}
\psi(A)=\phi(\widetilde{A}U^t)=(U\widetilde{A}^\tr \widetilde{A}U^\tr)^{1/2}=U(\widetilde{A}^\tr)^{1/2}\widetilde{A}U^\tr=U\begin{bmatrix}
    \Psi_1(A) & \frac{1}{\sqrt{2}}\Psi_3(A)\\
    \frac{1}{\sqrt{2}}\Psi_3(A) & \Psi_2(A)
\end{bmatrix}U^\tr=\\
=\frac{\Psi_1(A)-\Psi_2(A)}{\sqrt{2}}
E_1+\Psi_3(A)E_2+\frac{\Psi_1(A)+\Psi_2(A)}{\sqrt{2}}E_3,
\end{multline*}
where
$$
E_1=U\begin{bmatrix}
    \frac{1}{\sqrt{2}} & 0 \\
    0 & \frac{-1}{\sqrt{2}}
    \end{bmatrix}U^t,
    E_2=U\begin{bmatrix}
     0 & \frac{1}{\sqrt{2}} \\
    \frac{1}{\sqrt{2}} & 0
\end{bmatrix}U^t,
E_e=U\begin{bmatrix}
    \frac{1}{\sqrt{2}} & 0 \\
    0 & \frac{1}{\sqrt{2}}
\end{bmatrix}U^t,
$$
The matrices $E_1,E_2,E_3$ are orthogonal to each other and have euclidean length 1. If we set $$\Psi=\left(\frac{\Psi_1-\Psi_2}{\sqrt{2}},\Psi_3,\frac{\Psi_1+\Psi_3}{\sqrt{2}}\right),
$$
then $d(\Psi(A),\Psi(B))=d(\psi(A),\psi(B))$, for all $A,B\in \R^{2\times 3}$, so we may as well use $\Psi:\R^{2\times 3}\to \R^3$ instead of $\psi:\R^{2\times 3}\to \R^{3\times 3}$.

The matrix
$$
\begin{bmatrix}
    \Psi_1(A) & \frac{1}{\sqrt{2}}\Psi_3(A)\\
    \frac{1}{\sqrt{2}}\Psi_3(A) & \Psi_2(A)
\end{bmatrix}
$$
is is positive semidefinite, so the determinant $\Psi_1(A)\Psi_2(A)-\frac{1}{2}\Psi_3(A)^2$ and the trace $\Psi_1(A)+\Psi_2(A)$ are nonnegative. This implies that
the image of $\Psi$ lies in the cone in $\R^3$ given by inequalities $z\geq 0$ and $x^2+y^2\leq z^2$. In fact the image of of $\Psi$ is equal to this cone.

As an experiment we take $1$ million pairs of triangles in $\R^2$ where the coordinates of the vertices have a standard normal distribution. The distribution of the ratio $d(\gamma(A),\gamma(B))/d_{\E(2)}(A,B)$ is on the left in Figure~\ref{fig:1}.
The ratio ranged from $0.2268$
to $1.7321$. Theoretically
this ratio must always lie in the interval $[0,\sqrt{3}]$ theoretically. The mean of the distribution was $1.4043$ with a standard deviation of $0.2128$.
The natural logarithm of the ratio
had mean $0.3273$ and standard defiation $0.1601$.

The distribution of the ratio
$d(\Psi(A),\Psi(B))/d_{\E(2)}(A,B)$ is in Figure~\ref{fig:1} on the right. The ratio ranged from $1.0381$ to $1.4039$. Theoretically this ratio is always in the interval $[1,\sqrt{2}]$.
The mean of this distribution is
1.0381 and the standard deviation is $0.0510$. The natural logarithm of this ratio has mean $0.0362$ and
and standard deviation $0.469$.

\begin{figure}[h]\label{fig:1}
\caption{Distortion distribution for $\gamma$ (left) and $\Psi$ (right)}
\includegraphics[width=3in]{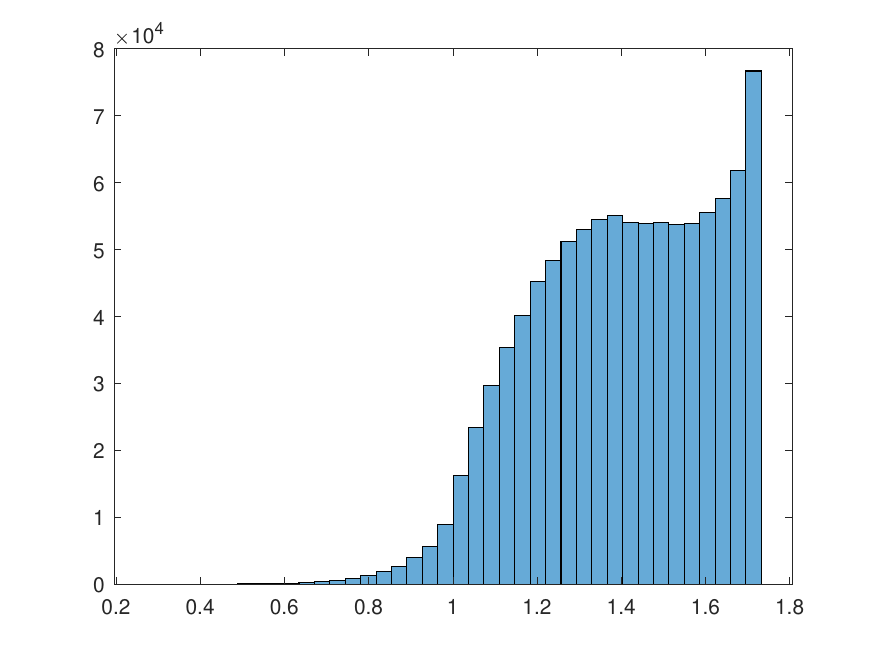}
\includegraphics[width=3in]{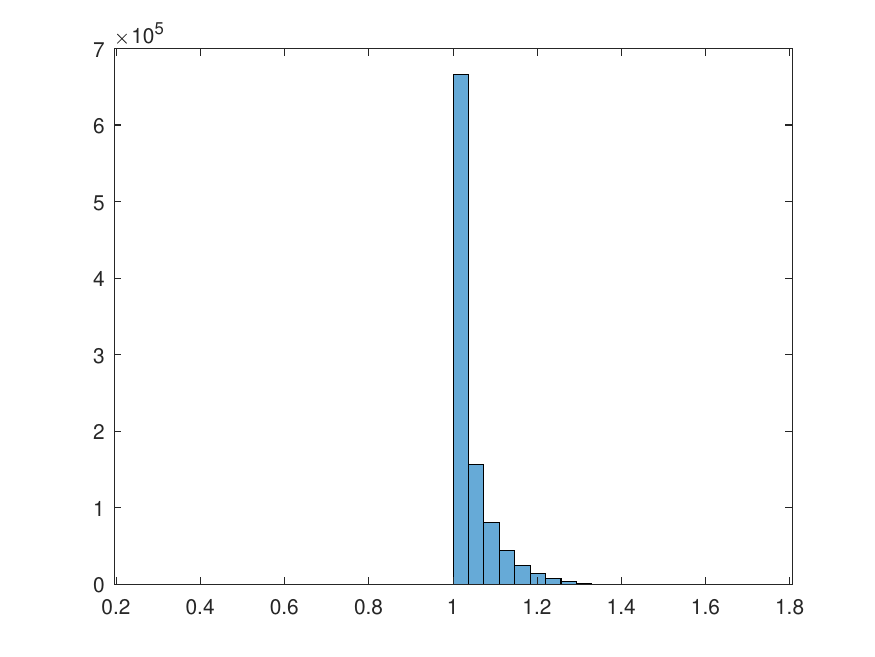}
\end{figure}

For a second experiment, we create a database of 5000 random triangles in $\R^2$. Again each coordinate of each vertex has a standard normal destribution. 
For each triangle $A$ in the database, we create another triangle $A_{\rm noisy}=A+E$
where $E$ is additive noise where each coordinate of each vertex has standard deviation $\varepsilon$.
Given $A_{\rm noisy}$ we verify if we can correctly classify the triangle by looking it up in the database. First, we choose a point $A'$ in the database with $d_{\E(2)}(A_{\rm noisy},A')$ minimal. In Figure~\ref{fig:2} we look at the ratio of points that are misclassified (i.e., $A'\neq A$). The blue graph plots error ratio (vertical axis) against $\varepsilon$ in the range $[0,0.03]$ (horizontal axis). 
Next we plot $\gamma$ to classify the points. We choose a point $A''$ such that $d(\gamma(A),\gamma(A'')$ is minimal. The red graph plots the ratio of misclassified points using $\gamma$ versus $\varepsilon$. We also choose a point $A'''$ such that
$d(\Psi(A),\Psi(A'''))$ is minimal and in the yellow graph we plot the ratio of misclassified points using $\Psi$ versus $\varepsilon$.
The graphs shows that using the map $\Psi$ to classify triangles performs as well as using the actual distance $d_{\E(n)}$.
If we use the map $\gamma$ to classify triangles, the performance is notably worse.

\begin{figure}[h]\label{fig:2}
\caption{Misclassification errors using $d_{\E(2)}$ (blue), $\gamma$ (red) and $\Psi$ (yellow)}
\includegraphics[width=4 in]{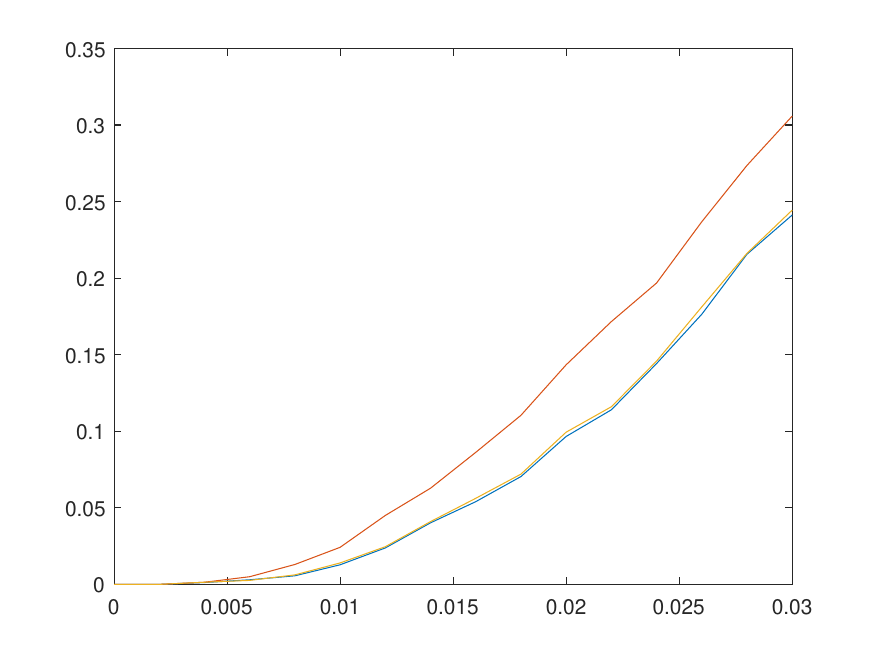}

\end{figure}

% References

\end{document}